\documentclass[reqno,12 pt]{amsart}%amsart

\usepackage{amsfonts,amscd,mathrsfs}
\usepackage{amsthm}
\usepackage{amssymb}
\usepackage{amsmath}
\usepackage{verbatim}
\usepackage{mathrsfs}
\usepackage{latexsym}
\usepackage{enumerate}
\usepackage{graphicx}
\usepackage[nobysame]{amsrefs}
\usepackage{color}
\usepackage{epsf}
\usepackage{geometry}
\usepackage{url}
\usepackage{hyperref}
\definecolor{dark-blue}{rgb}{0,0,0.6}
\definecolor{Purple}{rgb}{0.2,0,0.25}
\hypersetup{breaklinks=true,
pagecolor=white,
colorlinks=true,%false
citecolor=dark-blue,%{rgb}{0.2,0.2,1},%
filecolor=black,%
linkcolor=Purple,%
urlcolor=black
}
\newcommand{\bref}[1]{\textbf{\ref{#1}}} %bold font for any cross reference 
\newcommand{\beqref}[1]{\textbf{(\ref{#1})}} %bold font for any equation number

\theoremstyle{plain} 
\newtheorem{thm}{Theorem}[section]
\newtheorem{lem}[thm]{Lemma}

\newtheorem{cor}[thm]{Corollary}
\newtheorem{prop}[thm]{Proposition}

\theoremstyle{definition}
\newtheorem{remark}[thm]{Remark}
\newtheorem{Example}[thm]{Example}
\numberwithin{equation}{section}

\def\lv{\left\vert }
\def\rv{\right\vert}

\newcommand{\sC}{\mathscr{C}}

\newcommand{\dom}{\textnormal{dom}}

\newcommand{\sign}{\textnormal{sign}}
\newcommand{\abs}{\textnormal{abs}}

\newcommand{\R}{\mathbb{R}}

\newcommand{\N}{\mathbb{N}}


\begin{document}
\date{April 8, 2019}

\title[Fixed points of Legendre-Fenchel type transforms]{Fixed points of Legendre-Fenchel type transforms}

\author{Alfredo N. Iusem}
\address{Alfredo N. Iusem, IMPA - Instituto Nacional de Matem\'atica Pura e Aplicada, Estrada Dona Castorina 110, 
Jardim Bot\^anico, CEP 22460-320, Rio de Janeiro, RJ,  Brazil.}
\email{iusp@impa.br}
\author{Daniel Reem}
\address{Daniel Reem, Department of Mathematics, The Technion - Israel Institute of Technology, 3200003 Haifa, Israel.} 
\email{dream@technion.ac.il}
\author{Simeon Reich}
\address{Simeon Reich, Department of Mathematics, The Technion - Israel Institute of Technology, 3200003 Haifa, Israel.} 
\email{sreich@technion.ac.il}
\maketitle

\begin{abstract}
A recent result characterizes the fully order reversing operators acting on the class of 
lower semicontinuous proper convex functions in a real Banach space as certain linear deformations 
of the Legendre-Fenchel transform. Motivated by the Hilbert space version of this result and by the 
well-known result saying that this convex conjugation transform  has a unique fixed point (namely, 
the normalized energy function), we investigate the fixed point equation in which the involved operator 
is fully order reversing and acts on the above-mentioned class of functions. It turns out that this nonlinear 
equation is very sensitive to the involved parameters and can have no solution, a unique solution, 
or several (possibly infinitely many) ones. Our analysis yields a few by-products, such as results 
related to positive definite operators, and to functional equations and inclusions involving monotone operators. 
\end{abstract}

\section{Introduction}\label{sec:Introduction}
\subsection{Background:}\label{subsec:Background} 
Let $X$ be a real Hilbert space. Our goal is to solve the fixed point equation 
\begin{equation}\label{eq:f_Tf}
f(x)=\tau f^*(Ex+c)+\langle w,x\rangle +\beta,\quad x\in X, 
\end{equation}
where $\tau>0$, $c\in X$, $w\in X$ and $\beta\in\R$ are given, $E:X\to X$ is a 
given continuous linear invertible operator, and $f:X\to[-\infty,\infty]$ is 
the unknown function. Here 
\begin{equation}\label{eq:f^*}
f^*(x^*):=\sup\{\langle x^*,x\rangle-f(x): x\in X\}, \quad x^*\in X, 
\end{equation}
is the {\it Legendre-Fenchel transform} of the function $f$. 
This transform (which has many other names such as the {\it Legendre transformation}, or the {\it convex conjugation}, and sometimes has a form which is slightly different from \beqref{eq:f^*}) plays 
a central role in classical mechanics \cite{Arnold1989book}, thermodynamics \cite{ZiaRedishMcKay2009jour}, convex analysis \cite{Rockafellar1970book}, nonlinear analysis \cite{BurachikIusem2008book} 
and optimization \cite{Hiriart-UrrutyLemarechal1993book}. 
 
The motivation for discussing \beqref{eq:f_Tf} stems from several known results.  First, it is a more general version of the equation 
\begin{equation}\label{eq:SelfConjugate}
f=f^*,
\end{equation}
the solutions of which describe all the self-conjugate functions. It is well known 
(see \cite[Proposition 13.19, p. 225]{BauschkeCombettes2017book}, \cite[p. 106]{Rockafellar1970book}) 
that \beqref{eq:SelfConjugate} has a unique solution, namely  $f(x)=\frac{1}{2}\|x\|^2$, $x\in X$ (the ``normalized energy function''). 
This fact was mentioned briefly and without proof already in the pioneering work 
of Fenchel \cite[p. 73]{Fenchel1949} for $X=\R^n$, and later it was extended (with a proof) by Moreau \cite[Proposition 9.a]{Moreau1965jour} to general real Hilbert spaces.  

Second, \beqref{eq:f_Tf} is related to a relatively recent 
development in convex analysis. As shown in the pioneering work of Artstein-Avidan and 
Milman \cite[Theorem 7]{ArtsteinMilman2009} when $X=\R^n$, the right-hand side of \beqref{eq:f_Tf} is closely related to 
the characterization of fully order reversing operators $T$ acting on 
the class $\sC(X)$ of all proper, convex, and lower semi-continuous convex functions from $X$ to $\R\cup\{+\infty\}$. 
More precisely, if $T$ is fully order reversing, in the sense that $T$ is invertible and it reverses the point-wise order between functions 
in $\sC(X)$ 
(that is, if $f(x)\leq g(x)$ for all $x\in X$, then $(Tf)(u)\geq (Tg)(u)$ for all $u\in X$), 
and also $T^{-1}$ reverses the order between functions, then $T$ must have the form 
$T=T[E,c,w,\tau,\beta]$, where  
\begin{equation}\label{eq:T[]}
T[E,c,w,\tau,\beta](f)(x):=\tau f^*(Ex+c)+\langle w,x\rangle +\beta,\quad x\in X
\end{equation}
for some $\tau>0$, $c\in X$, $w\in X$, $\beta\in\R$ and an invertible linear operator 
$E:X\to X$. We think of $T$ as being a Legendre-Fenchel type transform since, up to certain inner and outer linear deformations, it indeed is this transform. The converse statement holds too: as can be verified directly, any operator of the form  $T=T[E,c,w,\tau,\beta]$ is a fully order reversing acting on $\sC(X)$.  
See \cite{Artstein-AvidanMilman2008,Artstein-AvidanMilman2010,
Artstein-AvidanMilman2011,Artstein-AvidanSlomka2012jour,BoroczkySchneider2008jour,SegalSlomka2012jour} 
for variations of this result regarding other classes of functions and geometric objects, still in a 
finite-dimensional setting. 

As shown in \cite[Theorem 2]{IusemReemSvaiter2015jour}, the above-mentioned characterization of fully order reversing operators 
holds also in the case of arbitrary real Banach spaces. Here one considers 
fully order reversing operators $T$ acting between $\sC(X)$ and the 
class of $\sC_{w^*}(X^*)$ of all weak$^*$ lower semicontinuous proper and convex 
functions from the dual $X^*$ to $\R\cup\{+\infty\}$. Now $f^*$ is defined on $X^*$ as in \beqref{eq:f^*} (where now $x^*\in X^*$ and $\langle x^*,x\rangle$ denotes $x^*(x)$ for all $x\in X$), 
the linear operator $E:X^*\to X^*$ is continuous on $X^*$, 
$c$ is a vector in $X^*$, $w$ is a vector in the canonical embedding of $X$ in $X^{**}$, 
$\tau$ is positive, and $\beta$ is real. 

Since such an operator $T$ acts between two different sets, 
one cannot speak of its possible fixed points. However, in the specific case 
where $X$ is a  Hilbert space the well-known strong correspondence between $X$ and $X^*$ enables us to identify them. This fact, when combined with the fact that for convex functions strong and weak lower semicontinuity coincide (and the same holds for the weak and weak$^*$ topologies in $X^*\cong X$), leads us to take $E:X\to X$, $c,w\in X$, and $T:\sC(X)\to\sC(X)$, and to conclude that \beqref{eq:f_Tf} describes the form of the most general fixed point equation 
of fully order reversing operators acting on $\sC(X)$. 

We mention now some works related to fixed point theory in the context of conjugates and convex analysis. 
One type of works we have already mentioned earlier, namely works which completely solve \beqref{eq:SelfConjugate} (for instance, \cite[Proposition 13.19, p. 225]{BauschkeCombettes2017book}, \cite[p. 73]{Fenchel1949}, \cite[Proposition 9.a]{Moreau1965jour} and \cite[p. 106]{Rockafellar1970book}). A second type of relevant works are \cite{AlvesSvaiter2011} and \cite{Svaiter2003}. More precisely, \cite{AlvesSvaiter2011} discusses \beqref{eq:SelfConjugate} in  
a rather general setting: $X$ is a nonempty set and the conjugation $f^*$ is abstract. It was proved in \cite[Theorem 1.1]{AlvesSvaiter2011} that in this case there exists at least one solution $f$
to \beqref{eq:SelfConjugate}. Variations of this theorem for more concrete settings  
(with the stronger result that now the fixed point must belong to the 
Fitzpatrick family of a maximal monotone operator) appeared earlier in 
 \cite[Theorem 2.4]{Svaiter2003} and later in \cite[Theorem 4.4]{AlvesSvaiter2011}. 
Neither \beqref{eq:f_Tf} nor the questions of uniqueness of solutions (to \beqref{eq:SelfConjugate}) 
and their classification have been considered in these papers.  

Another work which is somewhat related to our context is that of Rotem \cite{Rotem2012jour} in which self-polar functions on the positive ray were investigated. More precisely, the considered equation was $f=f^{\circ}$, where here $f^{\circ}$ is the polarity transform of $f$, a transform which was introduced in \cite[pp. 136-139]{Rockafellar1970book} and was extensively investigated in \cite{Artstein-AvidanMilman2011}. 
A complete  characterization of all the self-polar functions on the ray was presented in \cite[Theorem 5]{Rotem2012jour} and an example presented later \cite[p. 838]{Rotem2012jour} shows that this characterization fails already in a two-dimensional setting. Anyway, \beqref{eq:f_Tf} was not investigated in \cite{Rotem2012jour} (but it is an interesting open problem to investigate a version of \beqref{eq:f_Tf} in which $f^*$ is replaced by $f^{\circ}$).

\subsection{Contributions:} \label{subsec:Contribute}
The  main theorem of this paper is Theorem \bref{thm:Main} below which 
shows that the solution set of the nonlinear equation \beqref{eq:f_Tf} is very 
sensitive to the various parameters which appear in it, and classifies the possible solutions in many cases. 
In a nutshell, \beqref{eq:f_Tf} can have no solution, a unique solution, or many (possibly infinitely many) solutions. 
More precisely, the governing parameter seems 
to be the invertible linear operator $E$. If $E$ is positive definite, 
then there always exists a solution to \beqref{eq:f_Tf}, and this solution is quadratic 
and strictly convex. Sometimes uniqueness can also be established, and this existence and uniqueness 
result generalizes the well-known result mentioned in Subsection \bref{subsec:Background} that $f=\frac{1}{2}\|\cdot\|^2$ 
is the unique solution to   \beqref{eq:SelfConjugate}. On the other hand, if $E$ is not positive definite, then 
there can be several (perhaps infinitely many) solutions to \beqref{eq:f_Tf} or no solution at all, 
depending on the values of the other parameters which appear in \beqref{eq:f_Tf}. 
Moreover, in some cases there exist non-quadratic solutions.

To the best of our knowledge, \beqref{eq:f_Tf} has not been considered 
in the literature. Its analysis is somewhat technical and requires 
separation into several cases, according to the relevant parameters which appear in \beqref{eq:f_Tf}. 
 
Along the way we obtain a number of by-products which seem to be of independent interest. 
Among them, we mention Lemmas  \bref{lem:QL}--\bref{lem:Q-1=LQL} below, concerning certain functional inclusions 
and equations (see also Remark \bref{rem:LQL=Q-1}), and Corollary \bref{cor:Q^2=I} below regarding the uniqueness 
of square roots of the identity operator in the class of positive semidefinite linear operators.

\subsection{Paper layout:} After some preliminaries given in Section \bref{sec:Preliminaries}, 
we formulate the main classification theorem (Theorem \bref{thm:Main}) in Section \bref{sec:MainResult}. 
The tools needed in the proof of this theorem are developed in Sections \bref{sec:SimpleClaims}--\bref{sec:Nonuniqueness}, 
and the proof itself is presented in Section \bref{sec:ProofTheoremMain}. 
We finish the paper with Section \bref{sec:ConcludingRemarks} which contains several concluding remarks and open problems. 

\section{Preliminaries}\label{sec:Preliminaries}
We work with a real Hilbert space $X\neq\{0\}$ endowed with an inner product $\langle\cdot,\cdot\rangle$ 
and an induced norm $\|\cdot\|$. A function $f:X\to [-\infty,\infty]$ is called proper whenever $f(x)>-\infty$ for 
all $x\in X$ and, in addition, $f(x)\neq\infty$ for at least one point $x\in X$. 
The effective domain of $f:X\to[-\infty,\infty]$ is 
the set $\dom(f):=\{x\in X: f(x)\in \R\}$. 
The Fenchel-Legendre transform of $f:X\to [-\infty,\infty]$ is the function 
$f^*:X\to [-\infty,\infty]$ which is defined in \beqref{eq:f^*}. 

It is well known that $f^*$ is always convex and lower semicontinuous  
in the norm topology of $X$, 
and, in addition, that $f^*\equiv -\infty$ if and only if $f\equiv \infty$ (see, for instance, \cite{BauschkeCombettes2017book, VanTiel1984book} for the proofs of  many known facts from convex analysis which are mentioned here without proofs). 
The biconjugate of $f$ is the function  $f^{**}:X\to [-\infty,\infty]$  
defined by $f^{**}=(f^*)^*$. A well-known result, 
sometimes called the Fenchel-Moreau theorem \cite[Theorem 1.11, p. 13]{Brezis2011book}, 
says that $f=f^{**}$ whenever $f\in \sC(X)$. Here $\mathscr{C}(X)$ 
denotes the set of lower semicontinuous proper 
convex functions $f:X\to\R\cup\{+\infty\}$. 

We consider the pointwise order between functions, that is, we write $f\le g$ whenever $f(x)\le g(x)$ for all $x\in X$. It is well known and easily follows from \beqref{eq:f^*} that $f\leq g$ if and only if $f^*\geq g^*$. 
The subdifferential of $f$ at $x\in \dom(f)$ is the set $(\partial f)(x)$ defined by 
$(\partial f)(x):=\{x^*\in X: f(x)+\langle x^*,y-x\rangle\leq f(y),\quad\forall y\in X\}$. 
In general, $(\partial f)(x)$ is not necessarily a singleton and it can be  empty.  
If, however, $f$ is continuous (and, hence, finite everywhere), then $(\partial f)(x)$ is nonempty 
for all $x\in X$.  
We say that $f:X\to (-\infty,\infty]$ is strictly convex if for all $x,y\in \dom(f)$ satisfying $x\neq y$ 
and for all $\lambda \in (0,1)$ we have $f(\lambda x+(1-\lambda y)<\lambda f(x)+(1-\lambda) f(y)$. 

Given a linear and continuous operator
$E:X\to X$, the adjoint of $E$ is the operator $E^*:X\to X$ defined by 
the equation $\langle E^*a,b\rangle =\langle a,Eb\rangle$ for all $(a,b)\in
X^2$. It is well known that $E^*$ is  
continuous, and $\|E^*\|=\|E\|$. 

A self-adjoint operator is a continuous linear operator $E:X\to X$ 
satisfying $E=E^*$. Such an operator is also called symmetric. A self-adjoint operator $E:X\to X$ satisfying 
$\langle Ex,x\rangle\geq 0$ for all $x\in X$ 
is called positive semidefinite. A self-adjoint operator $E:X\to X$ satisfying 
$\langle Ex,x\rangle>0$ for all $0\neq x\in X$ is called positive definite 
or simply positive. We denote by $I:X\to X$ the identity operator. 

A set-valued operator on $X$ is a mapping  $A:X\to 2^X$, where $2^X$ is the set of all 
subsets of $X$. Such a mapping is frequently identified with the graph of $A$, that is, 
with the set $\textnormal{G}(A):=\{(x,y)\in X^2: y\in Ax\}$. We say that $A:X\to 2^X$ is single-valued  if $A(x)$
is a singleton for all $x\in X$. In this case we regard $A$ as an ordinary function from $X$ to $X$.  We say that $A:X\to 2^X$ is contained in $B:X\to 2^X$ whenever 
$\textnormal{G}(A)\subseteq \textnormal{G}(B)$, or, equivalently, when $Ax\subseteq Bx$ for each $x\in X$. 
For set-valued operators $A:X\to 2^X$ and $B:X\to 2^X$ and $x\in X$ we define $(A+B)(x):=\{a+b: a\in Ax,\,b\in Bx\}$ if $Ax\neq\emptyset$ and $Bx\neq\emptyset$ and $(A+B)(x):=\emptyset$ otherwise. The composition $BA$ (also denoted by $B\circ A$) is the 
operator from $X$ to $2^X$ defined by $(BA)(x):=\bigcup_{x'\in Ax}Bx'$, $x\in X$. 
This is an associative operation. The inverse of $A:X\to 2^X$, denoted by $A^{-1}$, is 
the set-valued operator the graph of which is $G(A^{-1})=\{(y,x): y\in Ax\}$. 
We call $A$ monotone whenever 
\begin{equation}\label{eq:Monotone}
\langle y_2-y_1, x_2-x_1\rangle \geq 0 \quad\,\forall (x_1,x_2)\in X^2,\, \forall 
y_i\in Ax_i, \,i=1,2.
\end{equation}
The set-valued operator $A$ is called strictly monotone if there is a strict inequality in \beqref{eq:Monotone}
whenever $x_1\neq x_2$.  
We say that $A$ is maximal monotone whenever it is  
monotone and there exists no  monotone operator $B: X\to 2^X$ such that 
$A\neq B$ and $A$ is contained in $B$. It is well known and straightforward to check that a set-valued operator  is maximal monotone if and only if its inverse is maximal monotone. In the sequel we make use of the well-known facts that the subdifferential of a proper function is monotone and any positive semidefinite linear operator $A:X\to X$ is maximal monotone (if $A$ is positive definite, then it is even strictly monotone). 
More information  regarding the theory of set-valued (monotone) operators can be found in \cite{BauschkeCombettes2017book,Brezis1973book,BurachikIusem2008book,Simons2008book}.  

A function $f:X\to (-\infty,\infty]$ is called \emph{at most quadratic} whenever 
\begin{equation}\label{eq:Quadratic}
f(x)=\frac{1}{2}\langle Ax,x\rangle+\langle b,x\rangle+\gamma, \quad x\in X, 
\end{equation}
for some  self-adjoint operator $A:X\to X$, 
a vector $b\in X$, and a real number $\gamma\in \R$. The operator $A$ is \emph{the leading coefficient} of $f$. 
We say that $f$ is \emph{quadratic} when it has the form \beqref{eq:Quadratic} with $A^*=A\neq 0$. 
A well-known fact (that follows from \cite[p. 29]{AmbrosettiProdi1993book}) which we use in Lemma \bref{lem:f_is_quadratic_and_strictly_convex} below is that if $f''$ exists at each point and is constant, then $f$ is at most quadratic.

A function $f:X\to(-\infty,\infty]$ is (Fr\'echet) differentiable at some $x\in X$ if $x\in\dom(f)$ and 
there exists a continuous linear functional $f'(x):X\to \R$ such that for every $h\in X$ sufficiently small, 
\begin{equation*}
f(x+h)=f(x)+f'(x)(h)+o(\|h\|).
\end{equation*}
The identification between $X$ and its dual $X^*$ allows us to write $\langle f'(x),h\rangle$ instead of $f'(x)(h)$.  The function $f$ is twice differentiable at $x\in X$ if $f':X\to X^*$ 
exists in a neighborhood of $x$ and is differentiable at $x$. The second derivative $f''(x)$ 
can be identified with a continuous and symmetric bilinear form acting from $X^2$ to $\R$.  
It is well known that any symmetric bilinear form $B:X^2\to\R$ can be written 
as $B(a,b)=\langle Aa,b\rangle$, $(a,b)\in X^2$, where $A:X\to X$ is a continuous and symmetric linear operator, and hence we  identify $f''(x)$ with the operator $A$ associated with it. If $f''$ exists and is continuous in a neighborhood of $x$, then 
$f$ has a second order Taylor expansion about $x$: 
\begin{equation}\label{eq:SecondOrder}
f(x+h)=f(x)+\langle f'(x),h\rangle+\frac{1}{2}\langle f''(x)h,h\rangle +o(\|h\|^2),\quad h\in X.
\end{equation}
According to a well-known fact, if $f$ is convex and differentiable at $x$, then $(\partial f)(x)=\{f'(x)\}$, in which case we write $(\partial f)(x)=f'(x)$.

\section{The classification theorem}\label{sec:MainResult}
The main result of this paper is the following classification theorem which 
analyzes the set of solutions of \beqref{eq:f_Tf} under various assumptions 
on the relevant parameters and on the class of allowed solutions. 
\begin{thm}\label{thm:Main}
Let $X$ be a real Hilbert space. Take $\tau>0$, $c\in X$, $w\in X$ and $\beta\in\R$. 
 Let $E:X\to X$ be an invertible and continuous linear operator. Consider 
the fixed point equation \beqref{eq:f_Tf}. Then the following statements hold:
\begin{enumerate}[(a)]%[label=(\alph*)]
\item\label{item:f is proper} Any solution $f:X\to[-\infty,\infty]$ of \beqref{eq:f_Tf} must be proper, convex 
and lower semicontinuous.
\item\label{item:StrictlyConvexQuadratic} If $E$ is positive definite, then there exists a strictly convex 
quadratic solution $f$ to \beqref{eq:f_Tf}, namely $f:X\to\R$ and it has the form \beqref{eq:Quadratic}. 
Its coefficients satisfy the following relations:
\begin{equation}\label{eq:FormQuadraticPositive}
\begin{array}{lll}
A&=&\sqrt{\tau}E,\\
b&=&\displaystyle{\frac{w+\sqrt{\tau}c}{1+\sqrt{\tau}}},\\
\gamma&=&\displaystyle{\frac{\beta(1+\sqrt{\tau})^2+\frac{1}{2}\sqrt{\tau}\langle  c-w,E^{-1}(c-w)\rangle}{(1+\sqrt{\tau})^2(\tau+1)}}.
\end{array}
\end{equation}
This solution is unique in the class of quadratic functions having a leading coefficient which is invertible. 
\item\label{item:StrictlyConvexUnique} Suppose that $E$ is positive definite and at least one of the following conditions holds: 
\begin{enumerate}[(i)]%[label=(\roman*)]
\item\label{item:c=0=w} $\tau=1$ and $c=w$,
\item\label{item:XisFiniteDimensional} $X$ is finite dimensional, $\tau\neq 1$, and 
 $f$ belongs to the class of functions from $X$ to $\R$ which are twice differentiable and their second derivative 
is continuous at the point 
\begin{equation}\label{eq:x_0}
x_0:=\frac{1}{1-\tau}(E^{-1}w-E^{-1}c).
\end{equation} 
\end{enumerate}
Then there exists a unique solution $f$ to \beqref{eq:f_Tf} (in the first case the uniqueness is in the class 
of all functions from $X$ to $[-\infty,\infty]$, and in the second case in the class of 
functions mentioned in Part \beqref{item:XisFiniteDimensional} above). In fact, this solution is quadratic and strictly convex and 
its coefficients satisfy \beqref{eq:FormQuadraticPositive}. 
\item\label{item:E is Not Positive Definite} If $E$ is not positive definite, then there are 
cases (which depend on $E$ and on the other parameters which appear in \beqref{eq:f_Tf}) in 
which \beqref{eq:f_Tf} does not have any solution, cases in which it has at least one solution, and cases 
in which  it has several solutions (possibly infinitely many) and some of these solutions are not quadratic. 
\end{enumerate}
\end{thm}
The proof of Theorem \bref{thm:Main} is given in Section \bref{sec:ProofTheoremMain} below. 
It is quite long and technical, and is based on several results presented in Sections \bref{sec:SimpleClaims}--\bref{sec:Nonuniqueness}.

\section{A few simple claims}\label{sec:SimpleClaims}
Here we recall without proofs known elementary facts which are needed later.   
\begin{lem}\label{lem:ConvexPositive}
Let $X$ be a real Hilbert space. Assume that function $f:X\to \R$ is at most quadratic. Then $f$ is convex if and only if its leading coefficient $A$ from \beqref{eq:Quadratic} is positive semidefinite; 
$f$ is strictly convex if and only if $A$ is positive definite. 
\end{lem}

\begin{lem}\label{lem:QuadConj}
Let $X$ be a real Hilbert space. Assume that $h:X\to \R$  has the form 
\begin{equation}
h(x)=\frac{1}{2}\langle Ax, x\rangle,\quad x\in X, 
\end{equation} 
for some positive semidefinite invertible operator $A:X\to X$. Then 
\begin{equation}
h^*(x^*)=\frac{1}{2}\langle A^{-1}x^*,x^*\rangle,\quad x^*\in X.
\end{equation}
\end{lem}

\begin{lem}\label{lem:AffineConjugate} 
Let $X$ be a real Hilbert space and let $E:X\to X$ be a continuous invertible linear operator, 
$c\in X ,w\in X$, $\beta\in\R$, and $\tau>0$ be given. 
Let $h:X\to[-\infty,\infty]$ and let $g:X\to[-\infty,\infty]$ be defined by 

\begin{equation*}
g(x):=\tau h(Ex+c)+\langle w,x\rangle +\beta,\quad x\in X. 
\end{equation*}
Then 
\begin{equation}
g^*(x^*)=\tau h^*(Hx^*+v)+\langle z,x^*\rangle+\rho, \quad x^*\in X, 
\end{equation}
where 
\begin{equation}\label{eq:Hvz_rho}
H:=\tau^{-1}(E^{-1})^*,\quad v:=-\tau^{-1}(E^{-1})^*w,\quad 
z:=-\tau^{-1}E^{-1}c,\quad
\rho:=\tau^{-1}(\langle w,E^{-1}c\rangle-\beta).
\end{equation}
\end{lem}

\section{General properties of solutions to \beqref{eq:f_Tf}}
The results presented in this section describe some properties 
that any solution $f$ to \beqref{eq:f_Tf} must satisfy. 
\begin{lem}\label{lem:f_in_sC(X)}
If $f:X\to [-\infty,\infty]$ solves \beqref{eq:f_Tf}, then $f\in \sC(X)$.
\end{lem}
\begin{proof}
 It is well known that $f^*$ is always convex and 
lower semicontinuous. Hence the right-hand side of \beqref{eq:f_Tf} is convex and 
lower semicontinuous, and therefore so is the left-hand side, namely $f$. It remains to be shown that $f$ is proper. If $f(x)=-\infty$ at some $x\in X$, 
then $f^*(Ex+c)=-\infty$ from \beqref{eq:f_Tf}. 
It is well known and easy to see that if $f^*$ is equal to $-\infty$ at some point, then $f\equiv \infty$ and $f^*\equiv -\infty$ 
leading to a contradiction, in view of  \beqref{eq:f_Tf}. By the same token
if $f\equiv \infty$, then 
$f^*\equiv -\infty$ and again \beqref{eq:f_Tf} leads to 
a contradiction. Therefore $f$ is proper.
\end{proof}

\begin{lem}\label{lem:f_functional}
Any solution $f:X\to [-\infty,\infty]$ to \beqref{eq:f_Tf} satisfies the following functional equations:
\begin{multline}\label{eq:f=T^2(f)}
f(x)=\tau^2 f\left(\tau^{-1}(E^{-1})^*Ex+\tau^{-1}(E^{-1})^*c-\tau^{-1}(E^{-1})^*w\right)\\
+\langle w-E^*E^{-1}c,x\rangle+\langle w, E^{-1}c\rangle-\langle E^{-1}c,c\rangle,\quad x\in X,
\end{multline}
and 
\begin{multline}\label{eq:f functional eq}
f(\tau E^{-1}E^*x+E^{-1}E^*E^{-1}w-E^{-1}c)\\
=\tau^2 f(x)+\langle w,\tau E^{-1}E^*x+E^{-1}E^*E^{-1}w-E^{-1}c\rangle
-\langle\tau^3 c,x\rangle+\beta(1-\tau^2),\quad \forall x\in X.
\end{multline}
In addition, if $E$ is self-adjoint, then 
\begin{multline}\label{eq:f functional eq E}
f(\tau x+E^{-1}w-E^{-1}c)\\
=\tau^2 f(x)+\langle w,\tau x+E^{-1}w-E^{-1}c\rangle
-\langle\tau^3 c,x\rangle+\beta(1-\tau^2),\quad \forall x\in X.
\end{multline}
\end{lem}
\begin{proof}
We first prove \beqref{eq:f=T^2(f)}. By applying $T=T[E,c,w,\tau,\beta]$  
on both sides of \beqref{eq:f_Tf}, we have $f=Tf=T^2f$. From \beqref{eq:f_Tf}, Lemma ~\bref{lem:AffineConjugate}, the equality $f=f^{**}$ (which holds since $f\in \sC(X)$ according to  Lemma ~\bref{lem:f_in_sC(X)}) 
and elementary calculations, we have 
\begin{multline}
(T^2f)(x)=\tau(Tf)^*(Ex+c)+\langle w,x\rangle+\beta\\
=\langle w,x\rangle+\langle w, E^{-1}c\rangle-\langle E^{-1}c,c\rangle-\langle E^*E^{-1}c,x\rangle\\
+\tau^2 f\left(\tau^{-1}(E^{-1})^*Ex+\tau^{-1}(E^{-1})^*c-\tau^{-1}(E^{-1})^*w\right)
\end{multline}
and this implies \beqref{eq:f=T^2(f)}. Now we prove \beqref{eq:f functional eq}. The equality $f=f^{**}$, equation \beqref{eq:f_Tf}, Lemma ~\bref{lem:AffineConjugate} 
and the change of variables $x\mapsto Ex+c$ imply that 
\begin{multline}\label{eq:f=f^**}
f(x)=(f^{*})^*(x)\\
=\left((1/\tau)f(E^{-1}(\cdot)-E^{-1}c))+
\langle -E^{-1}w/\tau,\cdot-c\rangle-\beta/\tau\right)^*(x)\\
=(1/\tau)f^*(\tau E^*x+E^*Ew)+\langle \tau c,x\rangle+\beta.
\end{multline}
From \beqref{eq:f_Tf} and elementary calculations it follows that 
\begin{multline*}
f^*(\tau E^*x+E^*Ew)=f^*(c+E(\tau E^{-1}E^*x+E^{-1}E^*E^{-1}w-E^{-1}c))\\
=(1/\tau)(f(\tau E^{-1}E^*x+E^{-1}E^*E^{-1}w-E^{-1}c)-\langle w,\tau E^{-1}E^*x+E^{-1}E^*E^{-1}w-E^{-1}c\rangle-\beta). 
\end{multline*}
This equality and \beqref{eq:f=f^**} imply \beqref{eq:f functional eq}, and since $E=E^*$, 
we get  \beqref{eq:f functional eq E} from
\beqref{eq:f functional eq}. 
\end{proof}

\begin{lem}\label{lem:Qq}
Let $f:X\to [-\infty,\infty]$ be a solution to \beqref{eq:f_Tf}.
Then there exist a linear operator $Q:X\to X$, a vector $q\in X$, and a real number $\theta$ such that 
\begin{equation}\label{eq:Quadratic<=f}
\frac{1}{2} \langle Qx,x\rangle+\langle q,x\rangle +\theta \le f(x),\quad \forall x\in X.
\end{equation}
In fact,
\begin{equation}\label{eq:Qq}
\begin{array}{l}
Q=\displaystyle{\frac{2\tau}{\tau+1}}E, \\
q=\displaystyle{\frac{\tau}{\tau+1}}\left(\displaystyle{\frac{1}{\tau}}w+c\right), \\
\theta=\displaystyle{\frac{\beta}{\tau+1}}. 
\end{array}
\end{equation}
\end{lem}
\begin{proof}
Lemma \bref{lem:f_in_sC(X)} implies that $f\in \sC(X)$. Hence it satisfies the Fenchel-Young inequality 
\begin{equation}\label{eq:Fenchel}
f^*(x^*)+f(x)\geq \langle x^*, x\rangle,\quad \forall x,x^*\in X. 
\end{equation}
In particular, this inequality holds for an arbitrary $x\in X$ and for $x^*:=Ex+c$. 
Since $f$ satisfies \beqref{eq:f_Tf} it follows that 
\begin{equation*}
f^*(x^*)=\frac{1}{\tau}f(x)-\left\langle \frac{1}{\tau}w,x\right\rangle-\frac{\beta}{\tau}.
\end{equation*}
This equality and  \beqref{eq:Fenchel} imply that for all $x\in X$  
\begin{multline}\label{eq:SubQuad}
f(x)\geq \frac{\tau}{\tau+1}\left(\left\langle \frac{1}{\tau}w+Ex+c, x\right\rangle+\frac{\beta}{\tau}\right)\\
=\frac{1}{2}\left\langle \frac{2\tau}{\tau+1}Ex,x\right\rangle+
\left\langle \frac{\tau}{\tau+1}\left(\frac{1}{\tau}w+c\right),x\right\rangle +\frac{\beta}{\tau+1}.
\end{multline}
This inequality implies \beqref{eq:Quadratic<=f} and \beqref{eq:Qq}. 
\end{proof}

\section{Two general lemmas and additional results}\label{sec:GeneraLemma}
In this section we present two general lemmas and a by-product of possible independent interest.
\begin{lem}\label{lem:QL}
Let $X$ be a real Hilbert space. Assume that $L:X\to 2^X$ is single-valued, invertible, strictly monotone and maximal monotone. 
Let $Q:X\to 2^X$ be a monotone operator.  If $I\subseteq QLQL$ or $I\subseteq LQLQ$, then $Q=L^{-1}$, and hence $Q$ is actually 
single-valued, invertible, strictly monotone and maximal monotone. 
\end{lem}
\begin{proof} 
Suppose first that $I\subseteq QLQL$. Hence, for all $x\in X$ there  exists $z\in (QL)(x)$ such that $x\in (QL)(z)$. 
Since $L$ is single-valued, we have $z\in Q(L(x))$ and $x\in Q(L(z))$. Therefore, if $x\neq z$, then, by the 
monotonicity of $Q$ and the strict monotonicity of $L$ we have 
\begin{equation*}
0\leq \langle z-x,L(x)-L(z)\rangle=-\langle x-z,L(x)-L(z)\rangle<0.
\end{equation*}
This contradiction implies that $x=z$. Since $x\in Q(L(z))$, it follows that $x\in Q(L(x))$. Therefore $L^{-1}(y)\in Q(y)$ for all $y\in X$,
and hence 
$L^{-1}\subset Q$. Since $L^{-1}$ is maximal monotone (a fact which follows directly from the assumption that $L$ is by maximal monotone), we conclude that $Q=L^{-1}$.

Assume now that $I\subseteq LQLQ$. We claim that this inclusion implies that $I\subseteq LQ$. 
Indeed, suppose to the contrary that for some $x\in X$ we have 
$x \notin (LQ)(x)$. 
Since $x\in (LQLQ)(x)$, there exists $z\in (LQ)(x)$ 
such that $x\in (LQ)(z)$. We have $x \neq z$ because  
$x \notin (LQ)(x)$. 
Since $z \in (LQ)(x)$ and $x \in (LQ)(z)$ and  $L$ is single-valued, 
there exist $x'\in Q(x)$ and $z' \in Q(z)$ such that $z = L(x')$ and $x = L(z')$, so that 
\begin{equation}\label{eq:Lxz}
x-z=L(z')-L(x')=-(L(x')-L(z')). 
\end{equation}
Since $z \neq x$, we get from 
\beqref{eq:Lxz} that 
$x' \neq z'$. 
Since $x' \in Q(x)$ and $z' \in Q(z)$, we get from the 
monotonicity of $Q$, \beqref{eq:Lxz}, and the strict monotonicity of $L$,
\begin{equation*}
0 \leq \langle x'-z',x-z \rangle = -\langle x'-z', L(x')-L(z') \rangle< 0. 
\end{equation*}
This is a contradiction, and so $I\subseteq LQ$. 
Thus, $x\in (LQ)(x)$ for all $x\in X$, namely $I\subset LQ$, so that $L^{-1}\subset Q$. Hence 
 $L^{-1}=Q$ because $L^{-1}$ is maximal monotone and $Q$ is monotone.   
\end{proof}

\begin{lem}\label{lem:Q-1=LQL}
Let $X$ be a real Hilbert space. Assume that $L:X\to 2^X$ is single-valued, invertible, strictly monotone and maximal monotone. 
If $Q:X\to 2^X$ is monotone, satisfies $Q(x)\neq\emptyset$ for all $x\in X$  and 
\begin{equation}\label{eq:LQLQ-1}
LQL = Q^{-1}, 
\end{equation}
then $Q=L^{-1}$.  In other words, \beqref{eq:LQLQ-1} has a unique solution in the set 
\begin{equation}\label{eq:Omega(X)}
\Omega(X):=\{Q:X\to 2^X: \,Q(x)\neq\emptyset\,\,\forall x\in X\,\,\textnormal{and}\,\,Q\,\,\textnormal{is monotone}\},
\end{equation}
and this solution is $Q=L^{-1}$.
\end{lem}
\begin{proof}
Take $x\in X$ and $y\in Q(x)$. Hence $x\in Q^{-1}y\subseteq (Q^{-1}Q)(x)$. In view of \beqref{eq:LQLQ-1} we get $x\in (Q^{-1}Q)(x)=(LQLQ)(x)$ for all $x\in X$, namely $I\subseteq LQLQ$. We conclude from Lemma \bref{lem:QL} that $Q=L^{-1}$. In other words, any solution $Q\in \Omega(X)$ to \beqref{eq:LQLQ-1} must coincide with $L^{-1}$. Finally, it is immediate to see that $L^{-1}$ (which belongs to $\Omega(X)$) does solve  \beqref{eq:LQLQ-1}.
\end{proof}

\begin{cor}\label{cor:Q^2=I}
If $X$ is a real Hilbert space and $Q:X\to X$ is a positive semidefinite linear operator satisfying $Q^2=I$, then $Q=I$.
\end{cor}
\begin{proof}
 It follows from Lemma \bref{lem:Q-1=LQL} with $L:=I$. 
\end{proof}

\begin{remark}\label{rem:Q^2=I FiniteDim}
When $\dim(X)=n\in\N$, Corollary \bref{cor:Q^2=I} is just a simple consequence of the fact that  $Q$, 
being self-adjoint, can be diagonalized. In other words, there exists a linear operator $U:X\to X$ 
satisfying $U^*U=I=UU^*$ such that $D:=U^*QU$ is a diagonal operator. Thus, $I=U^*U=U^*Q^2U=U^*QUU^*QU=(U^*QU)^2=D^2$. Since $Q$ is positive semidefinite, so is $D$. Thus $D=I$ and therefore $Q=UDU^*=I$. 
\end{remark}

\begin{remark}\label{rem:LQL=Q-1}
To the best of our knowledge, the functional equation \beqref{eq:LQLQ-1} has not been investigated so far. 
However, interestingly, versions of  \beqref{eq:LQLQ-1} can be found in several places in the literature. We mentioned next two of them. The first one appears in \cite[Equation (10), p. 1440]{RanReurings2004jour}, where  the considered equation is 
\begin{equation}\label{eq:MatrixEquation}
Q=P+\sum_{k=1}^mA_j^*(Q-C)^{-1}A_j,
\end{equation}
where $m$ and $n$ are natural numbers, $A_1,\ldots,A_m$ are arbitrary $n\times n$ matrices (not necessarily invertible), 
$P$ is an $n\times n$ positive definite matrix, $C$ is positive semidefinite, and the unknown $Q$ is an $n\times n$ matrix 
such that $Q-C$ is positive definite. This equation is inspired by the closely related matrix equation 
\cite[Equation 7.1.30. p. 95]{Sakhnovich1997book}  which appears in the study of extremal interpolation problems. 
Related matrix-type equations appear in \cite[Equation (5.1), p. 416]{PetruselRus2006jour} 
and \cite[Equation (11), p. 1440]{RanReurings2004jour}. Although \beqref{eq:LQLQ-1} and \beqref{eq:MatrixEquation}
have similarities, there are important differences between them (the classes in which the unknown $Q$ is sought and other differences). 

A second version of \beqref{eq:LQLQ-1} is simply the involution equation 
\begin{equation}\label{eq:h^2=I}
h^2=I,
\end{equation}
with unknown function $h$. It is equivalent to \beqref{eq:LQLQ-1} if we assume that the unknown $Q$ in \beqref{eq:LQLQ-1} 
is single-valued and make the change of variables $h=LQ$. There is a vast  literature on involutions in various settings. For instance, \cite{Artstein-AvidanMilman2008,ArtsteinMilman2009,Artstein-AvidanMilman2011,IusemReemSvaiter2015jour} discuss involutions in the context of operators acting on $\sC(X)$ or on closely related classes of functions and convex sets, and \cite[Chapter 11]{KuczmaChoczewskiGer1990book} discusses  involutions of functions from an interval to itself. However, we are not aware of works which investigate the equation $(LQ)^2=I$ in the context of Lemma \bref{lem:Q-1=LQL}.
\end{remark}

\section{Properties of the solutions to \beqref{eq:f_Tf}: the positive semidefinite and quadratic cases}\label{sec:PositiveDefiniteQuadratic}
This section presents properties of  the solutions to \beqref{eq:f_Tf} under additional assumptions on $E$ and/or $f$. 
\begin{lem}\label{lem:FormQuadratic}
Assume that $f:X\to\R$ has the form \beqref{eq:Quadratic}, where 
$A$ is invertible and positive semidefinite. If $f$ satisfies \beqref{eq:f_Tf}, then 
\begin{equation}\label{eq:FormQuadratic}
\begin{array}{lll}
 A &=&\tau E^*A^{-1}E,\\
(\tau E^* A^{-1} +I)b&=&w+\tau E^{*} A^{-1} c,\\
\gamma&=&\displaystyle{\frac{\beta+\langle\tau (c-b),\frac{1}{2} A^{-1}(c-b)\rangle}{\tau+1}}.
\end{array}
\end{equation}
On the other hand, if the coefficients of $f$ satisfy \beqref{eq:FormQuadratic}, where $\tau>0$ 
and $c,w\in X$ are given and $E:X\to X$ is a given self-adjoint invertible linear operator, then $f$ solves \beqref{eq:f_Tf} and we also have 
\begin{equation}\label{eq:tau(A^{-1}E)^2=I}
(\sqrt{\tau} A^{-1}E)^2=I=\left(\frac{1}{\sqrt{\tau}}E^{-1}A\right)^2. 
\end{equation}
\end{lem}
\begin{proof}
We can write $f(x)=h(x)+\langle b,x\rangle+\gamma$ where $h(x):=\frac{1}{2}\langle Ax,x\rangle$ 
for all $x\in X$. By using Lemma \bref{lem:QuadConj} and Lemma \bref{lem:AffineConjugate} we see that 
\begin{equation}\label{eq:f^*x^*}
f^*(x^*)=h^*(x^*-b)-\gamma=\langle x^*-b, \frac{1}{2} A^{-1}(x^*-b)\rangle-\gamma,\quad\forall x^*\in X. 
\end{equation}
Fix $x\in X$ and denote $x^*:=Ex+c$. Suppose first that $f$ solves 
\beqref{eq:f_Tf}. This equation, \beqref{eq:f^*x^*}, the facts that $A$ and hence $A^{-1}$ are self-adjoint, all imply that 
\begin{multline*}%\label{eq:Tf_Quad}
f(x)=\tau (\langle Ex +c -b, \frac{1}{2} A^{-1}(Ex+c-b)\rangle-\gamma)+\langle w,x\rangle+\beta\\
=\langle x,\frac{1}{2}\tau E^* A^{-1} Ex\rangle+\langle x, \tau E^*A^{-1}(c-b)+w\rangle 	
+\beta+\langle \tau(c-b),\frac{1}{2}A^{-1}(c-b)\rangle-\tau \gamma.
\end{multline*}
Since the left-most and right-most sides of this equation are quadratic functions and their leading coefficients 
are self-adjoint ($A$ by assumption, hence so is $\tau E^*A^{-1}E$), we can 
equate the coefficients of both functions and after doing this we obtain \beqref{eq:FormQuadratic}. On the other hand, suppose that $f$ satisfies \beqref{eq:FormQuadratic} and $E=E^*$. These assumptions and \beqref{eq:f^*x^*} imply that 
\begin{multline}\label{eq:Tf_Quad1}
\tau f^*(Ex+c)+\langle w,x\rangle+\beta\\
=\tau(\langle Ex+c-b, \frac{1}{2} A^{-1}(Ex+c-b)\rangle-\gamma)+\langle w,x\rangle+\beta\\
=\tau\langle Ex+c-b, \frac{1}{2} \tau^{-1}E^{-1}AE^{-1}(c-b)+\frac{1}{2}\tau^{-1}E^{-1}AE^{-1}Ex\rangle-\tau\gamma+\langle w,x\rangle+\beta\\
=\frac{1}{2}\langle x,AE^{-1}(c-b)\rangle+\frac{1}{2}\langle Ax,x\rangle+
\tau\langle c-b,\frac{1}{2}A^{-1}(c-b)\rangle+\langle c-b,\frac{1}{2} E^{-1}Ax\rangle\\
-\tau\gamma+\langle \tau E A^{-1}(b-c)+b,x\rangle+(\tau \gamma+\gamma-\tau\langle c-b,\frac{1}{2}A^{-1}(c-b)\rangle)\\
=\frac{1}{2}\langle Ax,x\rangle+\langle AE^{-1}(c-b)+\tau EA^{-1}(b-c),x\rangle+\langle b,x\rangle+\gamma\\
=\frac{1}{2}\langle Ax,x\rangle+\langle b,x\rangle+\gamma=f(x),
\end{multline}
using the fact that \beqref{eq:FormQuadratic} implies the equality $AE^{-1}=\tau EA^{-1}$ in the last but one equation. 
Therefore $f$ satisfies \beqref{eq:f_Tf}, as required. 
Finally, since $E=E^*$, it follows from the first equality in \beqref{eq:FormQuadratic} that the leftmost 
equality in \beqref{eq:tau(A^{-1}E)^2=I} holds, from  which the right equality in \beqref{eq:tau(A^{-1}E)^2=I} 
also follows by taking inverses.
\end{proof}

\begin{lem}\label{lem:FormQuadraticPositive}
Assume that $E$ is positive definite and that $f:X\to\R$ has the form \beqref{eq:Quadratic}, where 
$A$ is positive semidefinite and invertible.  
If $f$ satisfies \beqref{eq:f_Tf}, then its coefficients satisfy \beqref{eq:FormQuadraticPositive}. 
In particular, $A$ is actually positive definite.
\end{lem}

\begin{proof}
Since $A$ is positive semidefinite and since $f$ solves \beqref{eq:f_Tf} and has the form \beqref{eq:Quadratic}, 
 Lemma \bref{lem:FormQuadratic} implies \beqref{eq:FormQuadratic}. Denote $Q:=A$ and $L:=(\sqrt{\tau}E)^{-1}$. 
With this notation and the fact that $E$ is self-adjoint, we see that the first equation in \beqref{eq:FormQuadratic} is equivalent to 
the equation $LQL=Q^{-1}$. Since $E$ is continuous and invertible, $L$ is continuous and invertible. 
Since $E$ is positive definite and hence strictly monotone and maximal monotone, so is $L$. Since $A$ is positive semidefinite, $Q$ is monotone. Thus Lemma \bref{lem:Q-1=LQL} implies that $Q=L^{-1}$, namely $A=\sqrt{\tau}E$. 
Hence $A$ is positive definite. By substituting the previous expressions in the second equation of \beqref{eq:FormQuadratic} 
we see that $b=(w+\sqrt{\tau}c)/(1+\sqrt{\tau})$.  
Therefore $c-b=(c-w)/(1+\sqrt{\tau})$ and the expression for $\gamma$ in  
\beqref{eq:FormQuadraticPositive} follows.  
\end{proof}

\begin{lem}\label{lem:Q'q'}
Let $f:X\to [-\infty,\infty]$ be a solution to \beqref{eq:f_Tf}. If $E$ is invertible and positive semidefinite, 
then there exist a positive semidefinite invertible linear operator $Q':X\to X$, a vector $q'\in X$, and a real number $\theta'$ such that 
\begin{equation}\label{eq:f<=Quadratic}
f(x)\le \frac{1}{2} \langle Q'x,x\rangle+ \langle q',x\rangle +\theta',\quad \forall x\in X.
\end{equation} 
In fact, 
\begin{equation}\label{eq:Q'q'}
\begin{array}{l}
Q'=\frac{1}{2}(\tau+1)E^{-1}, \\
q'=\frac{1}{2}(c+w), \\
\theta'=\displaystyle{\frac{\beta}{\tau+1}}+\left\langle \displaystyle{\frac{1-\tau}{4}E^{-1}c}-\frac{1}{2}E^{-1}w,c\right\rangle+
\displaystyle{\frac{\left\langle E^{-1}(w+\tau c),w+\tau c\right\rangle}{4(\tau+1)}}. 
\end{array}
\end{equation}
\end{lem}

\begin{proof}
Since the conjugation reverses the order, by using Lemma \bref{lem:Qq} and taking conjugates on both sides of \beqref{eq:SubQuad}, it follows from Lemmas \bref{lem:AffineConjugate} 
and  Lemma \bref{lem:QuadConj} that for all $x^*\in X$,
\begin{multline}\label{eq:f^*<=}
f^*(x^*)\leq \frac{1}{2}\left\langle \frac{\tau+1}{2\tau}E^{-1}\left(x^*-\frac{\tau}{\tau+1}\left(\frac{1}{\tau}w+c\right)\right),
x^*-\frac{\tau}{\tau+1}\left(\frac{1}{\tau}w+c\right)\right\rangle-\frac{\beta}{\tau+1}\\
=\frac{1}{2}\left\langle \frac{\tau+1}{2\tau}E^{-1}x^*,x^*\right\rangle-\frac{1}{2}\left\langle E^{-1}\left(\frac{1}{\tau}w+c
\right),x^*\right\rangle\\
+\frac{1}{4}\left\langle \frac{\tau}{\tau+1}E^{-1}\left(\frac{1}{\tau}w+c\right),\frac{1}{\tau}w+c\right\rangle-\frac{\beta}{\tau+1}.
\end{multline}
Note that \beqref{eq:f_Tf} implies that $f^*(x^*)=(1/\tau)f(x)-\langle (1/\tau)w,x\rangle-(1/\tau)\beta$ 
for every $x\in X$ and for $x^*=Ex+c$. Combining this fact with \beqref{eq:f^*<=}, we obtain \beqref{eq:f<=Quadratic} and \beqref{eq:Q'q'}
after some algebra.
\end{proof}

\begin{cor}\label{cor:FiniteContinuous}
Let $f:X\to [-\infty,\infty]$ be a solution to \beqref{eq:f_Tf}, where the 
invertible linear operator $E:X\to X$ is assumed to be positive semidefinite. 
Then $f$ is finite and locally Lipschitz continuous everywhere. 
\end{cor}
\begin{proof}
Lemma \bref{lem:Qq} and Lemma \bref{lem:Q'q'} imply that $f$ is finite everywhere (hence proper), 
that is, its effective domain is $X$. Since, in addition, $f$ is convex and lower semicontinuous (Lemma \bref{lem:f_in_sC(X)}) and since $X$ is a Banach space, it follows from a well-known result that $f$ is continuous on $X$. As a matter of fact, either the continuity of $f$ or \beqref{eq:f<=Quadratic} imply that $f$ is locally bounded above everywhere and hence, by another well-known result in convex analysis, $f$ is locally Lipschitz continuous everywhere. 
\end{proof}

\begin{lem}\label{lem:Convex=StrictlyConvex}
Suppose that the invertible linear operator $E:X\to X$ 
from \beqref{eq:f_Tf} is positive semidefinite. Then any solution $f:X\to\R$ to \beqref{eq:f_Tf} which 
is at most quadratic must be strictly convex and, in particular, quadratic. 
\end{lem}
\begin{proof}
Lemma \bref{lem:f_in_sC(X)} ensures that $f$ is convex. Hence Lemma \bref{lem:ConvexPositive} implies 
that the leading coefficient $A$ of $f$ must be positive semidefinite, that is, $\langle Ax,x\rangle \geq 0$ for all $x\in X$. Assume to the contrary that $\langle Ay,y\rangle=0$ for some nonzero 
vector $y\in X$. This assumption, the fact that $f$ satisfies \beqref{eq:Quadratic}, 
and the definition of $f^*$ (in \beqref{eq:f^*}), all imply that for each $t\in \R$, 
\begin{equation}\label{eq:f^(y+b)}
f^*(y+b)\geq \langle y+b,ty\rangle-\left(\frac{1}{2}\langle A(ty),ty\rangle +\langle b,ty\rangle+\gamma\right)
=t\|y\|^2-\gamma.
\end{equation}
By taking the limit $t\to\infty$ in \beqref{eq:f^(y+b)} and using the assumption that $y\neq 0$ we find that 
$f^*(y+b)=\infty$. From \beqref{eq:f_Tf} with $x:=E^{-1}(y+b-c)$ and $x^*:=y+b$ it follows that $f(x)=\infty$. 
This contradicts Corollary \bref{cor:FiniteContinuous} which ensures that $f$ must be finite 
everywhere. Hence $A$ is positive definite. This fact and Lemma \bref{lem:ConvexPositive} imply that $f$ is strictly convex and quadratic.
\end{proof}

\begin{cor}\label{cor:AisInvertible}
Suppose that $X$ is finite dimensional and that $f:X\to\R$ is a solution to \beqref{eq:f_Tf} which is at most quadratic. If $E$ is positive semidefinite,  then the leading coefficient $A$ of $f$ must be invertible. 
\end{cor}

\begin{proof}
Lemma \bref{lem:Convex=StrictlyConvex} ensures that $f$ is strictly convex. 
Thus (Lemma \bref{lem:ConvexPositive})  $A$ is positive definite. Since $\dim(X)<\infty$, we conclude that $A$ is invertible. 
\end{proof}

\begin{lem}\label{lem:g_p}
Suppose that $f$ and $p$ solve \beqref{eq:f_Tf},  
where the invertible linear operator $E:X\to X$ is assumed to be positive semidefinite. 
Then there exists  a continuous function $g_{f,p}:X\to \R$ satisfying  
\begin{equation}\label{eq:f_represent}
f(x)=p(x)+g_{f,p}(x), \quad  x\in X, 
\end{equation}
and  
\begin{equation}\label{eq:g_p}
g_{f,p}(\tau x+E^{-1}w-E^{-1}c)=\tau^2 g_{f,p}(x),\quad x\in X.
\end{equation}
\end{lem} 
\begin{proof}
From Corollary \bref{cor:FiniteContinuous} we know that both $f$ and $p$ are finite 
and continuous everywhere. Thus, if we define 
$g_{f,p}:X\to [-\infty,\infty]$ by $g_{f,p}(x):=f(x)-p(x)$ for each $x\in X$, then $g_{f,p}$ 
is well defined, finite, and continuous everywhere. Lemma \bref{lem:f_functional} implies that both $f$ and $p$ satisfy 
\beqref{eq:f functional eq E}. By considering  the version of \beqref{eq:f functional eq E} with $f$, 
subtracting from it the version of \beqref{eq:f functional eq E} with $p$, 
and substituting $g_{f,p}$ in the corresponding places, we obtain \beqref{eq:g_p}. 
\end{proof}

\section{$E$ is positive definite: existence and partial uniqueness}\label{sec:ExistencePartialUniqueness} 

The following proposition shows the existence of a solution to \beqref{eq:f_Tf} 
when $E$ is positive definite. 
This solution is unique in the class of quadratic functions having a leading coefficient which is invertible.  

\begin{prop}\label{prop:Existence}
If the invertible linear operator $E:X\to X$ in \beqref{eq:f_Tf} is positive definite, 
then there exists a solution $p:X\to[-\infty,\infty]$ 
to \beqref{eq:f_Tf}. This solution is quadratic and strictly convex, and its  
coefficients are defined by \beqref{eq:FormQuadraticPositive}. Furthermore, $p$ is the 
unique function which solves \beqref{eq:f_Tf} in the class of quadratic functions  
having a leading coefficient which is invertible. 
\end{prop}

\begin{proof}
Let $p$ be the function defined by \beqref{eq:Quadratic} and having coefficients 
defined by \beqref{eq:FormQuadraticPositive}. 
Since $E$ is invertible, we have $A=\sqrt{\tau}E\neq 0$. Thus $p$ is quadratic. Direct 
calculations show that \beqref{eq:FormQuadratic} is satisfied. Since $E$ is positive definite, so is $A$. 
Therefore Lemma \bref{lem:FormQuadratic} implies that $p$ satisfies \beqref{eq:f_Tf}. 
Moreover, since $A$ is positive definite, Lemma \bref{lem:ConvexPositive} ensures that $p$ is strictly convex. 

Suppose now that $f$ is a quadratic function which solves \beqref{eq:f_Tf} 
and its leading coefficient $A$ (from \beqref{eq:Quadratic}) is invertible. Lemma \bref{lem:f_in_sC(X)} 
implies that $f$ is convex. Hence Lemma \bref{lem:ConvexPositive} ensures that $A$ is positive semidefinite 
and thus Lemma \bref{lem:FormQuadraticPositive}  implies that the coefficients of $f$ 
satisfy \beqref{eq:FormQuadraticPositive} (and $A$ is actually positive definite). 
Therefore $f$ coincides with $p$, namely there exists a unique solution 
to \beqref{eq:f_Tf} in the class of quadratic functions having a leading coefficient which is invertible. 
\end{proof}

\section{$E$ is positive definite: existence and uniqueness when $\tau=1$ and $w=c$}\label{sec:v=w=0}
In Proposition \bref{prop:c=0=w_tau=1} below we establish the uniqueness of solutions to \beqref{eq:f_Tf} when $E$ is positive definite, $\tau=1$ and $w=c$. An immediate consequence of this proposition is the classical fact mentioned in Subsection \bref{subsec:Background}  that the normalized energy function is the unique solution to \beqref{eq:SelfConjugate}. 
\begin{prop}\label{prop:c=0=w_tau=1}
Consider \beqref{eq:f_Tf} under the assumptions that $\tau=1$ and $w=c$ (in particular, when $w=c=0$), 
namely 
\begin{equation}\label{eq:f=f^(E)+beta}
f(x)=f^*(Ex+c)+\langle c,x\rangle+\beta, \quad x\in X.
\end{equation}
If, in addition, $E$ is positive definite, then there exists a unique solution 
$f:X\to[-\infty,\infty]$ to \beqref{eq:f_Tf}. This solution coincides with the  
strictly convex quadratic function $p$ from Proposition \bref{prop:Existence}. 
\end{prop}
\begin{proof}
Existence follows from Proposition \bref{prop:Existence}. 
As for uniqueness, suppose that some function $f:X\to[-\infty,\infty]$ solves \beqref{eq:f_Tf}. 
Lemma \bref{lem:f_in_sC(X)} implies that $f\in\sC(X)$. 
 From Corollary \bref{cor:FiniteContinuous} it follows that $f$ is  
 continuous (and finite). As is well known, this fact implies that $(\partial f)(x)\neq \emptyset$ for each $x\in X$. The change of variables $x\mapsto Ex+c$, elementary calculations and \beqref{eq:f_Tf} lead to 
\begin{equation}\label{eq:fL}
f(Bx+d)=f^*(x)+\langle B^*c,x\rangle+\langle c,d\rangle+\beta,\quad x\in X, 
\end{equation}
where $B:=E^{-1}$ and $d:=-Bc$. Now we apply the subdifferential operator to both sides of \beqref{eq:fL} and we use the following known facts: 
\begin{itemize}
\item[i)]
$\partial f^*=(\partial f)^{-1}$;  
\item[ii)]  $B^*=B$ (because $E$ is positive definite);
\item[iii)] $(\partial f_d)(z)=(\partial f)(z+d)$ for all $z\in X$  where $f_d(z):=f(z+d)$ for all $z\in X$;
\item[iv)] the subdifferential of the sum of two lower semicontinuous proper convex functions is equal to the sum of the subdifferentials when the effective domain of one of the functions is the whole space; 
\item[v)] if $g$ is convex and differentiable, then $\partial g(x)=\{g'(x)\}$ for all $x\in X$;
\item[vi)] $\partial (g\circ B)=B^*\circ(\partial g)\circ B$ for all $g\in \sC(X)$. 
\end{itemize}
We conclude from (i)-(vi) that  $BQL=Q^{-1}+Bc$, where $Q:X\to 2^X$ and $L:X\to X$ are the operators defined 
by $Q(x):=(\partial f)(x)$ and $L(x):=Bx+d$ for all $x\in X$. 
Therefore, by adding $d$ to both sides of this equation and recalling that $d=-Bc$, we arrive at the equation $LQL=Q^{-1}$. 
Since $E$ is positive definite and invertible, so is $B$. Thus $L$ is the translation by a vector of an  invertible, 
strictly monotone and maximal monotone operator, and so the same holds for $L$.
Since $Q$ is clearly monotone,  Lemma \bref{lem:Q-1=LQL} can be used to conclude that $Q=L^{-1}$. 
Hence $(\partial f)(x)=E(x-d)$ for each $x\in X$, and so $\partial f$ is single-valued and continuous. We conclude that $f'(x)=(\partial f)(x)$ and, as a result, $f'(x)=E(x-d)=Ex+c$ for all $x\in X$. 

Let $x\in X$ be fixed  and let $g:\R\to \R $ be the function defined by $g(t):=f(tx)$ for all $t\in\R$. 
Then $g$ is differentiable and  for all $t\in\R$
\begin{equation*}
g'(t)=\langle x, f'(tx)\rangle=\langle x, E(tx)+c\rangle=\langle Ex,x\rangle t+\langle c,x\rangle. 
\end{equation*}
Thus for every $x\in X$, we have 
\begin{equation}\label{eq:f(x)=Quadratic}
f(x)=g(1)=g(0)+\int_{0}^1 g'(t)dt=\frac{1}{2} \langle Ex, x\rangle+\langle c,x\rangle+f(0).
\end{equation}
Therefore $f$ is quadratic and its leading coefficient is $E$, which is an invertible and positive definite operator. 
Since $f$ solves \beqref{eq:f_Tf}, Proposition \bref{prop:Existence} implies that $f$ coincides with the strictly 
convex quadratic function $p$ defined there, as claimed. (Note: from \beqref{eq:f(x)=Quadratic} the linear coefficient of $f$ is $c$, but from \beqref{eq:Quadratic} and \beqref{eq:FormQuadraticPositive} it should be $b$; there is no contradiction since \beqref{eq:FormQuadraticPositive} and $w=c$ imply that $b=c$.)  
\end{proof}

\section{$E$ is positive definite: existence and uniqueness when both $\tau\neq 1$ and $f''$ exists and is continuous 
at a point}\label{sec:f''} 
In this section we show that  if $E$ is positive definite, $X$ is 
finite-dimensional and $\tau\neq 1$, then there exists a unique solution to \beqref{eq:f_Tf} in the class of functions having a second derivative which is continuous at a 
certain point (the finite-dimensionality of 
$X$ is only needed in Proposition \bref{prop:UniqueTwiceDiffQuadratic}  and not in Lemma \bref{lem:f_is_quadratic_and_strictly_convex}). 
\begin{lem}\label{lem:f_is_quadratic_and_strictly_convex}
Suppose that $\tau\neq 1$ and that the invertible operator $E:X\to X$ 
is positive definite. Assume that $f:X\to[-\infty,\infty]$  solves \beqref{eq:f_Tf}. 
If $f$ is twice differentiable on $X$ and its second derivative 
is continuous at the point $x_0:=(1/(1-\tau))(E^{-1}w-E^{-1}c)$, then $f$ must be strictly convex and quadratic. 
\end{lem}
\begin{proof}
Let $p:X\to \R$ be the strictly convex and quadratic solution to \beqref{eq:f_Tf} from Proposition ~\bref{prop:Existence}. 
Lemma ~\bref{lem:g_p} implies the existence of a function $g_{f,p}:X\to\R$ 
such that \beqref{eq:f_represent} and \beqref{eq:g_p} hold. 
Since $p$ is quadratic, it has a continuous second derivative. Thus $g_{f,p}:=f-p$ is 
twice differentiable and its second derivative is continuous at the point $x_0$. 

Let $x_1:=E^{-1}w-E^{-1}c$. Then $x_1=(1-\tau)x_0$. By differentiating \beqref{eq:g_p} we conclude that 
$\tau g_{f,p}'(\tau x+x_1)=\tau^2 g_{f,p}'(x)$ for each $x\in X$. 
Since $\tau\neq 0$, it follows that $g_{f,p}'(\tau x+x_1)=\tau g_{f,p}'(x)$ for each $x\in X$. 
A second differentiation and a division by $\tau$  yields  
\begin{equation}\label{eq:g_p''}
g_{f,p}''(\tau x+x_1)=g_{f,p}''(x)\quad \forall x\in X.   
\end{equation}
Assume first that $\tau\in (0,1)$. We fix $x$ and use \beqref{eq:g_p''} iteratively to obtain
\begin{multline}\label{eq:g_p''Iterative}
g_{f,p}''(x)=g_{f,p}''(\tau x+x_1)=g_{f,p}''(\tau (\tau x+x_1)+x_1)=g_{f,p}''(\tau^2 x+\tau x_1+ x_1)\\
      =\ldots=g_{f,p}''(\tau^m x+(\tau^{m-1}+\tau^{m-2}+\ldots+1)x_1)  
\end{multline}
for all $m\in\N$. By taking the limit $m\to\infty$ in \beqref{eq:g_p''Iterative} and using the 
assumptions that $\tau\in (0,1)$, that $g_{f,p}''$ is continuous at $x_0$ and that  $x_1/(1-\tau)=x_0$, we obtain 
$g_{f,p}''(x)=g_{f,p}''(x_1/(1-\tau))=g_{f,p}''(x_0)$ for all $x\in X$. Hence $g_{f,p}''$ is constant. 
Thus $g_{f,p}$ is at most quadratic.  

Consider now the case $\tau>1$. This case  follows from \beqref{eq:g_p''} again by first 
denoting $y:=\tau x +x_1$ and then observing that this notation and \beqref{eq:g_p''} lead to
\begin{equation}\label{eq:g_p''(y)}
g_{f,p}''(y)=g_{f,p}''(\tau^{-1}y-\tau^{-1}x_1)=g_{f,p}(\alpha y+y_1)\quad \forall y\in X, 
\end{equation}
where $\alpha:=\tau^{-1}$ and $y_1:=-\tau^{-1}x_1$. 
The equality $x_1=(1-\tau)x_0$ and the definition of $y_1$ 
imply that $y_0:=y_1/(1-\alpha)=x_0$. Our assumption on $g_{f,p}''$ thus 
implies that $g_{f,p}''$ is continuous at $y_0$. 
This observation,  \beqref{eq:g_p''(y)}, and the inequality $0<\alpha <1$ imply, 
 as in  \beqref{eq:g_p''Iterative} and the derivation after it, 
 that $g_{f,p}$ must be at most quadratic. Therefore $g_{f,p}$ is at most quadratic in both cases $\tau\in (0,1)$ and $\tau\in (1,\infty)$. 

Since $p$ is quadratic and $f=p+g_{f,p}$, it follows that $f$ is at most quadratic. 
Since $f$ solves \beqref{eq:f_Tf}, Lemma \bref{lem:f_in_sC(X)} implies that $f\in \sC(X)$. Hence $f$ is convex and since  
it solves \beqref{eq:f_Tf} we can conclude from Lemma ~\bref{lem:Convex=StrictlyConvex} 
that $f$ is strictly convex. Hence its leading coefficient cannot be the zero operator, and thus $f$ is   quadratic. 
\end{proof}

\begin{prop}\label{prop:UniqueTwiceDiffQuadratic}
Assume that the Hilbert space $X$ is finite dimensional. Given a positive definite and invertible linear operator $E:X\to X$, 
a positive number $\tau\neq 1$, and two vectors $c,w\in X$, consider the class of functions $f:X\to\R$ 
which are twice differentiable and their second derivative is continuous at the point $x_0:=(1/(1-\tau))(E^{-1}w-E^{-1}c)$. 
Then there exists a unique solution $f$ to \beqref{eq:f_Tf} in this class. In fact, this unique solution is 
the quadratic and strictly convex solution $p$ from Proposition ~\bref{prop:Existence}. 
\end{prop}

\begin{proof}
The quadratic function $p$ from Proposition \bref{prop:Existence} solves \beqref{eq:f_Tf} according to this proposition and 
it belongs to the considered class of functions. This shows the existence of a solution to \beqref{eq:f_Tf} in this class of functions. 
For uniqueness, suppose that $f$ belongs to the considered class of functions and that it satisfies \beqref{eq:f_Tf}. 
Lemma ~\bref{lem:f_is_quadratic_and_strictly_convex} implies that 
$f$ is strictly convex and quadratic. Hence $f$ satisfies \beqref{eq:Quadratic}. Corollary \bref{cor:AisInvertible} 
implies that its leading coefficient is invertible. It follows from Proposition  \bref{prop:Existence} 
that $f=p$ (and hence for $g_{f,p}$ from the proof of Lemma \bref{lem:f_is_quadratic_and_strictly_convex} we have $g_{f,p}\equiv 0$), as claimed. 
\end{proof}

\section{$E$ is not positive semidefinite: Nonexistence}\label{sec:Nonexistence}
When $E$ is not positive semidefinite, then even simple special cases of \beqref{eq:f_Tf} 
may have no solutions.
\begin{lem}\label{lem:FunctionalEq_f}
If $0\neq w\in X$, then there exists no solution $f\in \sC(X)$ 
to the functional equation 
\begin{equation}\label{eq:FunctionalEq_f}
f(x)=f(x+w)+\langle w,x\rangle,\quad x\in X.  
\end{equation}
\end{lem}

\begin{proof} 
Suppose to the contrary that some $f\in \sC(X)$ satisfies \beqref{eq:FunctionalEq_f}. 
Since $f$ is proper, there is a point $x_0\in X$ such that $f(x_0)\in \R$. 
Consider the function $\phi:\R\to[-\infty,\infty]$ defined by $\phi(t):=f(x_0+tw)$ for each $t\in\R$. 
By putting $x:=x_0+tw$, $t\in\R$ in  \beqref{eq:FunctionalEq_f} we see that $\phi$ satisfies the functional equation 
\begin{equation}\label{eq:FunctionalPhi}
\phi(t)=\phi(t+1)+\delta t+\rho,\quad t\in \R
\end{equation}
where $\delta:=\|w\|^2>0$ and $\rho:=\langle w,x_0\rangle$.

We claim that $\phi$ must be finite everywhere. Indeed, first 
$\phi(t)>-\infty$ for all $t\in \R$ because $f$ is proper. 
It remains to show that $\phi(t)<\infty$ for all $t\in \R$.   
By the choice of $x_0$ we have $\phi(0)=f(x_0)\in \R$. By setting $t:=0$ in \beqref{eq:FunctionalPhi} 
we see that $\phi(1)=\phi(0)-\rho$ and therefore $\phi(1)\in \R$. 
By putting $t:=-1$ in \beqref{eq:FunctionalPhi} we obtain that $\phi(-1)=\phi(0)-\delta+\rho$ and hence also $\phi(-1)\in \R$.
Induction and \beqref{eq:FunctionalPhi} yield $\phi(m)\in \R$ 
for all integers $m$. From the convexity of $f$ it follows that $\phi$ is convex, 
and thus, since any $t\in \R$ satisfies $t\in [m,m+1]$ for some integer $m$, 
we have  $\phi(t)\leq \max\{\phi(m),\phi(m+1)\}<\infty$, and hence $\phi$ is indeed finite everywhere. 

Since $\phi$ is convex, it has a left derivative $\phi_-$ which is an increasing function on $\R$.  
By taking the left derivative on both sides of 
\beqref{eq:FunctionalPhi} one sees that $\phi_{-}(t)=\phi_{-}(t+1)+\delta$ 
for each $t\in\R$. In particular,  
$\phi_{-}(0)=\phi_{-}(1)+\delta>\phi_{-}(1)$, a contradiction with the above-mentioned 
fact that $\phi_-$ is increasing. Hence \beqref{eq:FunctionalEq_f} cannot have 
any solution $f\in\sC(X)$.
\end{proof}

\begin{prop}\label{prop:Nonexistence}
If $w\neq 0$, then no $f:X\to [-\infty,\infty]$ solves the equation
\begin{equation}\label{eq:Nonexistence}
f(x)=f^*(-x)+\langle w, x \rangle, \quad x\in X.
\end{equation}
In addition, if $c\neq 0$, then no $f:X\to [-\infty,\infty]$ satisfies the equation  
\begin{equation}\label{eq:Nonexistence2}
f(x)=f^*(-x+c), \quad x\in X.
\end{equation}
\end{prop}
\begin{proof}
Consider first \beqref{eq:Nonexistence} and suppose to the contrary that it has a solution $f:X\to [-\infty,\infty]$. Lemma \bref{lem:f_in_sC(X)} implies that $f\in \sC(X)$ and Lemma \bref{lem:f_functional} (equation \beqref{eq:f=T^2(f)}) implies that $f$ is a solution to  \beqref{eq:FunctionalEq_f}. 
This contradicts Lemma \bref{lem:FunctionalEq_f} and proves that no $f:X\to [-\infty,\infty]$ can satisfy \beqref{eq:Nonexistence}. Now consider \beqref{eq:Nonexistence2} and assume to the contrary that some $f:X\to [-\infty,\infty]$ solves it. Let $F:X\to [-\infty,\infty]$ be defined by $F(x):=f(x+c)$ for all $x\in X$. A direct calculation based on Lemma \bref{lem:AffineConjugate} implies that $F$ solves \beqref{eq:Nonexistence}  with $w:=-c\neq 0$, a contradiction to what we established above. Hence no $f:X\to [-\infty,\infty]$ solves \beqref{eq:Nonexistence2}.  
\end{proof}

\section{$E$ is not positive semidefinite: existence}
\label{sec:E_is_not_positive_existence} 
Below we describe a case in which $E$ is not positive definite but \beqref{eq:f_Tf} does have solution. 
For the sake of a simpler exposition, we present this proposition only for $n$-dimensional spaces, $n\in\N$, but we mention that the result can be extended to separable Hilbert spaces $X$ and invertible continuous linear operators $E:X\to X$ which can be diagonalized using a unitary operator.

Given a self-adjoint  $E:\R^n\to\R^n$ (not necessarily positive semidefinite), we 
identify $E$ with its associated symmetric matrix. Since $E$ can be diagonalized,  we can write $E=UDU^{-1}$, where $U$ is unitary and $D$ is a diagonal matrix with diagonal elements $d_1, \dots , d_n$. We denote by $\abs(D)$ the diagonal matrix with diagonal elements $\lv d_1\rv, \dots, \lv d_n\rv$ and by $\sign(D)$ the diagonal matrix with diagonal elements $\sign(d_1),\dots,\sign(d_n)$. Since $E$ is invertible, so are $D$ and $\abs(D)$. In particular, $d_i\neq 0$ for all $i\in\{1,\ldots,n\}$. Define 
\begin{equation}\label{e1}
A:=\sqrt{\tau}U\abs(D)U^{-1}.
\end{equation}  
We have the following result.

\begin{prop}\label{prop:E_is_not-positive_existence} 
Consider \beqref{eq:f_Tf} and suppose that the invertible linear operator $E$ is self-adjoint (but possibly
not positive semidefinite) and take $A$ as in \beqref{e1}. 
Suppose that the set of solutions $x$ to the equation $(\tau EA^{-1}+I)x=w+\tau EA^{-1}c$ is nonempty. Let $b$ any such a solution and define 
\begin{equation*}
\gamma:=\frac{\beta+\frac{1}{2}\tau\langle c-b,A^{-1}(c-b)\rangle}{\tau+1}.
\end{equation*} 
Let $f:X\to \R$ be defined by $f(x):=\frac{1}{2}\langle Ax,x\rangle+\langle b,x\rangle+\gamma$ for all $x\in X$. Then $f$ solves \beqref{eq:f_Tf}. In particular, there exists a solution to \beqref{eq:f_Tf} if $w=-\tau EA^{-1}c$ 
(more particularly, when $c=w=0$), and this solution is 
\begin{equation}\label{eq:w=-tauEA(-1)c}
f(x):=\frac{1}{2}\langle Ax,x\rangle+\frac{\beta+\frac{1}{2}\tau\langle c,A^{-1}c\rangle}{\tau+1}, \quad x\in X. 
\end{equation}
\end{prop}
\begin{proof}
The assumptions on $A$ and $E$ imply that $E=UDU^{-1}=E^*$, that 
\begin{equation*}
\tau EA^{-1}=\tau UDU^{-1}(1/\sqrt{\tau})U(\abs(D))^{-1}U^{-1}=\sqrt{\tau}U\sign(D)U^{-1},
\end{equation*} 
and that $\tau EA^{-1}E=\sqrt{\tau}U\sign(D)U^{-1}UDU^{-1}=\sqrt{\tau}U\abs(D)U^{-1}=A$. 
The assumption on $b$ implies that $(\tau EA^{-1}+I)b=w+\tau EA^{-1}c$. These equalities and the definition 
of $\gamma$ imply that \beqref{eq:FormQuadratic} is satisfied. Since $A$ is invertible and positive definite and 
$E$ is self-adjoint, Lemma \bref{lem:FormQuadratic} implies that $f$ solves \beqref{eq:f_Tf}. 
Finally, if $w=-\tau EA^{-1}c$, then $w+\tau EA^{-1}c=0$, and hence $b:=0$ satisfies $0=(\tau EA^{-1}+I)b$. 
For this $b$ we have $\gamma=(\beta+\frac{1}{2}\tau\langle c,A^{-1}c\rangle)/(\tau+1)$. We conclude from the previous
argument that the function $f:X\to\R$ defined by \beqref{eq:w=-tauEA(-1)c} solves \beqref{eq:f_Tf}. 
\end{proof}

\begin{remark} 
Define $E:\R^n\to\R^n$ as $E:=-I$. Let $\tau:=1$, $c:=0$ and $\beta:=0$. Given $w\in X$, if $w=0$, then 
Proposition \bref{prop:E_is_not-positive_existence} ensures that \beqref{eq:f_Tf} has a solution. However, if $w\neq 0$, then Proposition \bref{prop:Nonexistence} ensures that \beqref{eq:f_Tf} 
no $f:X\to [-\infty,\infty]$ solves \beqref{eq:f_Tf}. A similar conclusion holds when $E:=-I$, $w:=0$, $\tau:=1$, $\beta:=0$ and $c=0$ or $c\neq 0$. This phenomenon is another manifestation to the sensitivity of \beqref{eq:f_Tf} with respect to the various parameters which appear in it. 
\end{remark}

\section{$E$ is not positive semidefinite: non-uniqueness}\label{sec:Nonuniqueness}
This section shows that there can be several (actually infinitely many) solutions to \beqref{eq:f_Tf} 
when $E$ is not positive semidefinite. We first consider the case of quadratic solutions (Example \bref{ex:skew}) 
and then of non-quadratic ones (Example \bref{ex:log}). 
\begin{Example}\label{ex:skew}
Suppose that $X=\R^2$. Assume that $E(x_1,x_2)=(x_2,-x_1)$ for all $x=(x_1,x_2)\in X$, 
that $\tau=1$, that $c=w=0$, and that $\beta=0$. 
In other words,  \beqref{eq:f_Tf} becomes 
\begin{equation}\label{eq:f_Tf_infinitely_many}
f(x_1,x_2)=f^*(x_2,-x_1),\quad (x_1,x_2)\in X. 
\end{equation}
Let $B=(b_{ij})_{i,j=1,2}$ be an arbitrary symmetric positive semidefinite matrix having real entries and a 
determinant which is equal to 1. We look at $B$ as an operator in $\R^2$,  namely $B(x)=Bx$. It is straightforward to verify that the function $f:X\to\R$ defined by $f(x):=\frac{1}{2}\langle Bx,x\rangle$ for all $x\in X$ solves \beqref{eq:f_Tf_infinitely_many}. 

We note that using the above result we can find solutions to the equation 
\begin{equation}\label{eq:f_Tf_R2n}
f(x_1,x_2,\ldots,x_{2n-1},x_{2n})=f^*(x_2,-x_1,\ldots,x_{2n-1},-x_{2n}),\quad (x_i)_{i=1}^{2n}\in \R^{2n},
\end{equation}
where $1<n\in\N$ is fixed. Indeed, for each $i\in \{1,\ldots,n\}$ let $B_i$ be any $2\times 2$ symmetric 
positive semidefinite matrix with  $\det(B_i)=1$ let $g_i:\R^2\to\R$ be defined 
by $g_i(x_1,x_2):=\frac{1}{2}\langle B_i(x_1,x_2),(x_1,x_2)\rangle$ for each $(x_1,x_2)\in \R^2$ 
and let $g:\R^{2n}\to\R$ be defined by $g(x):=\sum_{i=1}^ng_i(x_{2i-1},x_{2i})$ for every $x=(x_i)_{i=1}^{2n}\in \R^{2n}$. 
Then, the well-known formula for the conjugate of a direct sum gives $g^*(x)=\sum_{i=1}^n g^*_i(x_{2i-1},x_{2i})$. 
From the previous paragraph $g_i$ solves \beqref{eq:f_Tf_infinitely_many} for each $i\in\{1,\ldots,n\}$. Thus 
\begin{multline*}
g^*(x_2,-x_1,\ldots,x_{2n},-x_{2n-1})=\sum_{i=1}^n g^*_i(x_{2i},-x_{2i-1})
=\sum_{i=1}^n g_i(x_{2i-1},x_{2i})=g(x_{i})_{i=1}^{2n}
\end{multline*}
for every $(x_i)_{i=1}^{2n}\in \R^{2n}$. Hence $g$ solves \beqref{eq:f_Tf_R2n}. 
\end{Example}

\begin{Example}\label{ex:log}
Consider the equation 
\begin{equation}\label{eq:QuadraticLog}
f(x)=f^*(-x), \quad x\in X,
\end{equation}
namely \beqref{eq:f_Tf} with $E=-I$, $c=w=0$ and $\beta=0$. Assume first that $X=\R$. 
According to Proposition \bref{prop:E_is_not-positive_existence}, the function $f_1:\R\to\R$ 
defined by $f_1(x):=\frac{1}{2}x^2$ for all $x\in X$ solves \beqref{eq:QuadraticLog}.  
However, as already mentioned in \cite[p. 106]{Rockafellar1970book}, the function $f_2:\R\to[-\infty,\infty]$ defined by 
\begin{equation*}
f_2(x):=\left\{\begin{array}{ll}
\infty, &  x\in (-\infty,0],\\
-\frac{1}{2}-\log(x), &  x\in (0,\infty),
\end{array}\right.
\end{equation*}
also solves \beqref{eq:QuadraticLog}. A simple verification shows that 
two other types of solutions to \beqref{eq:QuadraticLog} are, respectively, 
\begin{equation*}
f_{3}(x):=\left\{\begin{array}{ll}
\infty,  &  x\in (-\infty,0),\\\\
0, & x\in [0,\infty),
\end{array}\right.
\end{equation*}
and 
\begin{equation}\label{eq:f_4,lambda}
f_{4,\lambda}(x):=\left\{\begin{array}{ll}
\displaystyle{\frac{\lambda}{2}x^2},  &  x\in (-\infty,0],\\\\
\displaystyle{\frac{1}{2\lambda}x^2}, & x\in [0,\infty),
\end{array}\right.
\end{equation}
where $\lambda>0$ is arbitrary. 
In addition, if some $f:X\to[-\infty,\infty]$ solves \beqref{eq:QuadraticLog}, then so  does 
the function $f_{-}(x):=f(-x)$, $x\in X$; indeed, using Lemma \bref{lem:AffineConjugate} and 
the fact that $f$ satisfies \beqref{eq:QuadraticLog}, we have $f_{-}^*(x)=f^*(-x)=f(x)=f_{-}(-x)$ for every $x\in X$. 
Now suppose that $X=\R^n$ for some integer $n>1$. Let $g:X\to [-\infty,\infty]$ be defined by 
$g(x):=\sum_{i=1}^n g_i(x_i)$ for all $x=(x_i)_{i=1}^n\in X$, 
where $g_i\in \{f_1,f_2,{f_{2}}_-,f_3,{f_3}_-\}\cup\{f_{4,\lambda}: \lambda\in (0,\infty)\}$ for every $i\in\{1,\ldots,n\}$. 
Using the well-known formula for the conjugate of a direct sum we conclude that any one of the functions $g$ mentioned 
above solves \beqref{eq:QuadraticLog}. 
\end{Example}

\section{Proof of Theorem \bref{thm:Main}}\label{sec:ProofTheoremMain}
\begin{proof}[{\bf Proof of Theorem \bref{thm:Main}}]
Part \beqref{item:f is proper} follows from Lemma \bref{lem:f_in_sC(X)}.
Part \beqref{item:StrictlyConvexQuadratic} follows from Proposition \bref{prop:Existence}. 
Part \beqref{item:StrictlyConvexUnique}\beqref{item:c=0=w} follows from Proposition \bref{prop:c=0=w_tau=1}. 
Part \beqref{item:StrictlyConvexUnique}\beqref{item:XisFiniteDimensional} follows from Proposition \bref{prop:UniqueTwiceDiffQuadratic}. 
Part \beqref{item:E is Not Positive Definite} follows from Propositions \bref{prop:Nonexistence} (non-existence),  
Proposition \bref{prop:E_is_not-positive_existence} (existence of certain quadratic solutions), 
Example \bref{ex:skew} (infinitely many quadratic solutions), and Example \bref{ex:log} (non-quadratic solutions).  
\end{proof}

\section{Concluding Remarks and open problems}\label{sec:ConcludingRemarks}
We conclude the paper with the following remarks. 

\begin{remark}\label{rem:OpenProblems}
At the moment it is not clear whether \beqref{eq:f_Tf} always has a unique solution (namely, the quadratic function with coefficients defined in \beqref{eq:FormQuadraticPositive}) whenever $E$ is positive definite.  Actually, even in the simple cases,  where $0<\tau\neq 1$ and $f(x)=\tau f^*(x)$ 
for all $x\in X$ or $f(x)=f^*(x+c)$ for each $x\in X$, where $c\neq 0$ is fixed, it is not clear whether the above-mentioned  quadratic function is the unique solution to one of these equations, and we suspect that non-uniqueness can hold. As we saw in Sections \bref{sec:Nonexistence}--\bref{sec:Nonuniqueness}, 
even more substantial  complications arise when $E$ is not positive definite, and the task of giving a complete 
description of the structure of the solutions to \beqref{eq:f_Tf} in this case, as a function of the various parameters which appear in \beqref{eq:f_Tf}, seems to be out of reach now even 
in the finite-dimensional case (an interesting open issue in this direction is whether there can be cases where 
the number of solutions to \beqref{eq:f_Tf} is finite, but greater than one). 
\end{remark}

\begin{remark}\label{rem:FullyLegendre}
Suppose that we look for solutions $f$ of \beqref{eq:f_Tf} in the class of differentiable (Fr\'echet or G\^ateaux)  functions. The change of variables $y:=Ex+c$ and \beqref{eq:f_Tf} implies that $f^*$ is also differentiable,  and Lemma \bref{lem:f_in_sC(X)} 
ensures that $f$ is convex, proper and lower semicontinuous. Functions $f:X\to\R$ having the property that they are convex, proper, 
lower semicontinuous, and they and their conjugates are G\^ateaux differentiable were investigated recently 
in \cite{ReemReich2017accep}. They are called \emph{fully Legendre}. If $f$ is fully Legendre, then its gradient $f'$ is invertible and satisfies $(f')^{-1}=(f^*)'$ (see \cite[Lemma 3.6]{ReemReich2017accep}). Hence, given $x^*\in X$, the function $F(x):=f(x)-\langle x^*,x\rangle$, 
$x\in X$, is  proper, lower semicontinuous, convex, and G\^ateaux differentiable on $X$. 
These conditions ensure that $F'(x)=(\partial F)(x)$.  Moreover, $F'$ vanishes at the (unique) point 
$x(x^*):=(f')^{-1}(x^*)$ and hence $x(x^*)$ is a global minimizer of $F$. 
The previous discussion and \beqref{eq:f^*} imply that 
$f^*(x^*)=\sup_{x\in X}[-F(x)]=-F(x(x^*))=\langle x^*,(f')^{-1}(x^*)\rangle-f((f')^{-1}(x^*))$. 
We conclude that $f$ satisfies the following functional-differential equation:
\begin{equation}\label{eq:FunctionalDifferential}
f(x)=\langle Ex+c,(f')^{-1}(Ex+c)\rangle-f((f')^{-1}(Ex+c))+\langle w,x\rangle+\beta,\quad x\in X.
\end{equation}
In particular, the functions mentioned in Propositions \bref{prop:c=0=w_tau=1}, \bref{prop:E_is_not-positive_existence}, 
Example \bref{ex:skew}, and in \beqref{eq:f_4,lambda} solve  \beqref{eq:FunctionalDifferential}. 
Conversely, if we look for solutions $f$ of \beqref{eq:FunctionalDifferential} which are fully Legendre, then 
the previous discussion implies that $f$ solves \beqref{eq:f_Tf} too. 
\end{remark}

\begin{remark}\label{rem:SemiDefinite=Definite}
An interesting corollary of Lemma \bref{lem:Convex=StrictlyConvex} is the following assertion: 
An invertible positive semi-definite linear operator $A$ acting from a real Hilbert space $X$ into itself 
must be positive definite. Indeed, let $f(x):=\frac{1}{2}\langle Ax,x\rangle$, $x\in X$. Then $f^*(x)=\frac{1}{2}\langle A^{-1}x,x\rangle$ for each $x\in X$ by Lemma \bref{lem:QuadConj}. Hence $f^*(Ax)=f(x)$ for every $x\in X$, that is, $f$ solves \beqref{eq:f_Tf} with $E=A$, $c=w=0$, $\beta=0$. Lemma \bref{lem:Convex=StrictlyConvex} ensures that $f$ is strictly convex, and from 
Lemma \bref{lem:ConvexPositive} we conclude that $A$ is positive definite.
\end{remark}

\newpage
%\vspace{0.1cm}
\noindent{\bf Acknowledgments}\vspace{0.1cm}\\
\noindent  Part of the work of the second author was done in 2013, while he was in IMPA - The National Institute of Pure and 
Applied Mathematics, Rio de Janeiro, Brazil, and this is an opportunity for him to thank a 
special postdoc fellowship from IMPA  (``P\'os-doutorado de Excel\^encia'').  The third author was partially supported 
by the Israel Science Foundation (Grant 389/12), by the Fund for the Promotion of Research at the 
Technion and by the Technion General Research Fund. The second author wants to thank Michael Cwikel for a discussion concerning a general aspect related to the paper. All the authors are grateful to the referee for several useful comments.

%\bibliographystyle{amsplain}
%\bibliography{biblio}


\end{document}